\documentclass[a4paper,11pt]{article}
\usepackage[T1]{fontenc}
\usepackage[utf8]{inputenc}
\usepackage{latexsym}
\usepackage{amsmath,amsthm,amssymb,amsfonts}
\usepackage{mathtools}
\usepackage{scalerel}
\usepackage[english]{babel}
\usepackage{tikz}
\usetikzlibrary{patterns, matrix}
\tikzstyle{NE-lines}=[pattern=north east lines, pattern color=black!45]
\usepackage{cancel}
\usepackage{booktabs}

\newtheorem{theorem}{Theorem}
\newtheorem{proposition}[theorem]{Proposition}
\newtheorem{lemma}[theorem]{Lemma}
\newtheorem{corollary}[theorem]{Corollary}

\theoremstyle{definition}

\newtheorem*{remark*}{Remark}
\newtheorem{example}{Example}

\newenvironment{DrawPerm}
{
\begin{tikzpicture}[scale=0.45, baseline=20pt]
}
{
\end{tikzpicture}
}

\newcommand{\fillPerm}[3]{
\draw [semithick] (0.001,0.001) grid ({#2},{#3});
\foreach \y [count=\xi] in {#1};
	\foreach \x in {1,...,\xi};
		\foreach \y [count=\x] in {#1} \filldraw (\x,\y) circle (6pt);
}

\newcommand{\meshBox}[2]{\fill[NE-lines] #1 rectangle #2;}

\newcommand{\etal}{et~al.}
\newcommand{\catalan}{\mathfrak{c}} 

\newcommand{\Sym}{\mathrm{Sym}} 
\newcommand{\identity}{\mathsf{id}} 

\newcommand{\reverse}{^{\mathsf{R}}}

\newcommand{\Sort}{\mathrm{Sort}} 

\newcommand{\mapsigma}[1]{\mathcal{S}^{#1}}

\newcommand{\sorted}{\mathrm{Sorted}} 
\newcommand{\fert}[1]{\mathrm{fert}^{#1}} 

\title{Dynamical aspects of $\sigma$-machines}
\author{Giulio Cerbai\\
\vspace{0.1cm}\\
\emph{University of Iceland}\\
\emph{Science Institute, 107 Reykjavik}\\
\texttt{giulio@hi.is}}
\date{}

\begin{document}

\thispagestyle{empty}
\maketitle


\begin{abstract}
The $\sigma$-machine was recently introduced by Cerbai, Claesson and Ferrari as a tool to gain a better insight on the problem of sorting permutations with two stacks in series. It consists of two consecutive stacks, which are restricted in the sense that their content must at all times avoid a certain pattern: a given~$\sigma$, in the first stack, and~$21$, in the second. Here we prove that in most cases sortable permutations avoid a bivincular pattern~$\xi$. We provide a geometric decomposition of $\xi$-avoiding permutations and use it to count them directly. Then we characterize the permutations with the property that the output of the $\sigma$-avoiding stack does not contain~$\sigma$, which we call effective. For $\sigma=123$, we obtain an alternative method to enumerate sortable permutations. Finally, we classify $\sigma$-machines and determine the most challenging to be studied.
\end{abstract}

\section{Introduction}\label{section_intro}

The problem of sorting permutations using a stack, together with its many variants, has been extensively studied in the literature (see~\cite{Bo} for a survey on stack sorting disciplines). A stack is a data structure equipped with two operations: \emph{push}, which adds the next element of the input to the stack, at the top; and \emph{pop}, which extracts from the stack the most recently pushed element. To sort a permutation means to produce a sorted output, that is, the identity permutation. As it is well known, an optimal algorithm to sort permutations using a stack consists in using a \emph{greedy} procedure on an \emph{increasing} stack. Here, the word increasing indicates that elements inside the stack are always mantained in increasing order, reading from top to bottom; and greedy means that a push is always performed, as long as it does not violate the previous constraint of monotonicity. If it does, a pop is performed instead. There are several interesting variations of classical stack sorting. The most straightforward generalization to two stacks was proposed by Julian West~\cite{We2}, who considered two passes through an increasing stack. Later on, Smith~\cite{Sm} considered a decreasing stack followed by an increasing one. Cerbai, Cioni and Ferrari~\cite{CeCiF} generalized Smith's device by allowing multiple decreasing stacks in series.

While things are usually rather manageable when considering a deterministic process, they become much more challenging if the problem is considered in its full generality, that is, if the sorting algorithm is not fixed. For instance, sorting with two consecutive stacks is still considered among the hardest open problems in combinatorics. Pierrot and Rossin~\cite{PR} proved that the problem of deciding whether a given permutation is sortable is polynomial-time solvable. It is known that sortable permutations form a class with infinite basis, but both the basis and the enumeration of the class remain unknown. In attempt to obtain a better understanding of this problem, Cerbai, Claesson and Ferrari~\cite{CeClF} introduced $\sigma$-stacks and $\sigma$-machines. The underlying idea is rather simple. An increasing stack can be equivalently regarded as a $21$-avoiding stack (briefly $21$-stack), meaning that at any point the $21$-stack is not allowed to contain an occurrence of the pattern $21$ (once again reading its content from top to bottom). The $\sigma$-stack is obtained by replacing~$21$ with a given permutation~$\sigma$. Finally, the $\sigma$-machine consists in making a pass through a $\sigma$-stack, followed by a pass through a $21$-stack, each time using a greedy procedure. This is an extremely natural generalization of West's device, which is obtained by simply choosing $\sigma=21$. The combinatorics underlying $\sigma$-machines proved to be extremely rich, and led to several follow-up papers. The first two~\cite{CeClF,CeClFS} were devoted to the study of the so-called $\sigma$-sortable permutations, for specific choices of $\sigma$. The $(\sigma,\tau)$-machine~\cite{BCKV,Be} generalized the $\sigma$-machine to pairs of patterns. The current author~\cite{Ce} extended the domain of $\sigma$-machines to Cayley permutations.

The main objective of this paper is to continue the study of $\sigma$-machines in a twofold way. Initially we focus on $\sigma$-sortable permutations and show that in most cases they avoid a certain bivincular pattern. Then we initiate the study of $\sigma$-sorted permutations, shifting our focus to the outputs produced by the $\sigma$-stack.

More in detail, in Section~\ref{section_preliminaries} we introduce the necessary definitions and tools.

In Section~\ref{section_bivincular_result} we define a bivincular pattern $\xi$ and show that $\sigma$-sortable permutations are always $\xi$-avoiding, unless $\sigma=12\ominus\beta$ is the skew-sum of $12$ with a non-empty $231$-avoding permutation $\beta$ (Theorem~\ref{theorem_bivincular_pattern_gen_result}). At present, this is one of the few known results that cover a family of $\sigma$-machines under the same umbrella. We also give a geometrical description of permutations avoiding $\xi$, and use it to enumerate them. The results obtained in this section suggest that the most challenging $\sigma$-machines to be studied arise when $\sigma=12\ominus\beta$.

In Section~\ref{section_sorted_perm_fertilities} we analyze some dynamical aspects of $\sigma$-machines by regarding a $\sigma$-stack as an operator $\pi\mapsto\mapsigma{\sigma}(\pi)$. This approach has been adopted recently in several works related to stack sorting~\cite{BCKV,Be,Ce,DZ}. Here we introduce some new properties and provide the first related results. We define $\sigma$-fertilties and $\sigma$-sorted permutations, two notions whose classical analogues are largely investigated. Then we characterize the permutations $\sigma$ having the property that the $\sigma$-stack does not create occurrences of $\sigma$, when the input is $\sigma$-sortable (Corollary~\ref{corollary_sorted_class}). We call these permutations \emph{effective}. Finally, we use $\sigma$-sorted permutations, together with the fact that $123$ is effective, to obtain an alternative method of enumerating $123$-sortable permutations. 

In Section~\ref{section_final_remarks} we use all the results obtained so far to initiate a classification of $\sigma$-machines. We end the paper with some questions and open problems.

\section{Preliminaries}\label{section_preliminaries}

Denote by $\Sym_n$ the set of permutations of length $n$, written in one-line notation, and let $\Sym=\bigcup_{n\ge 1}\Sym_n$. Let $\pi=\pi_1\cdots\pi_n\in\Sym_n$ and $\sigma\in\Sym_k$, for some $k\le n$. Then $\pi$ \emph{contains} $\sigma$ if there is a subsequence $\pi_{i_1}\cdots\pi_{i_k}$, with $i_1<\dots<i_k$, that is \emph{order isomorphic} to~$\sigma$; that is, such that $\pi_{i_j}<\pi_{i_{\ell}}$ if and only if $\sigma_j<\sigma_{\ell}$. Less formally, if the entries of $\pi_{i_1}\cdots\pi_{i_k}$ have the same relative order as $\sigma$.  In this case, we write $\pi_{i_1}\cdots\pi_{i_k}\simeq\sigma$ and say that $\pi_{i_1}\cdots\pi_{i_k}$ is an \emph{occurrence} of $\sigma$ in $\pi$. Otherwise, we say that $\pi$ \emph{avoids}~$\sigma$. Denote by $\Sym(\sigma)$ the set of permutations that avoid~$\sigma$ and let $\Sym_n(\sigma)=\Sym(\sigma)\cap\Sym_n$. Similarly, let $\Sym(B)$ denote the set of permutations avoiding every pattern in $B$, where $B\subseteq\Sym$. Pattern containment is a partial order on $\Sym$ and the resulting poset is the \emph{permutation pattern poset}. As a matter of fact, we will often write $\pi\ge\sigma$, if $\pi$ contains $\sigma$, and $\pi\not\ge\sigma$, otherwise. A \textit{permutation class}, or simply \emph{class}, is a downset of the permutation pattern poset. A class is uniquely determined by the set of minimal elements in the complementary upset, which is called its \textit{basis}. For instance, $\Sym(B)$ is a class with basis $B$. Note also that, if $\sigma\le\pi$, then $\Sym(\sigma)\subseteq\Sym(\pi)$, a fact that will be used repeatedly throughout this paper. We refer the reader to~\cite{Bev} for a brief introduction to permutation patterns and to~\cite{Bo2} for a more comprehensive textbook.

A \emph{bivincular} pattern~\cite{BMCDK} of length $k$ is a triple $\xi=(\sigma,X,Y)$, where $\sigma\in\Sym_k$ and $X,Y$ are subsets of $\lbrace 0,1,\dots,k \rbrace$. An occurrence of the bivincular pattern $\xi=(\sigma,X,Y)$ in a permutation $\pi$ is a classical occurrence $\pi_{i_1}\cdots\pi_{i_k}$ of $\sigma$ such that:
\begin{itemize}
\item $i_{x+1}=i_x+1$, for each $x \in X$;
\item $j_{y+1}=j_{y}+1$, for each $y \in Y$,
\end{itemize}
where $\lbrace\pi_{i_1},\dots,\pi_{i_k}\rbrace=\lbrace j_1,\dots,j_k \rbrace$, with $j_1<\cdots<j_k$; by conventions, $i_0=j_0=0$ and $i_{k+1}=j_{k+1}=n+1$. The set $X$ identifies constraints of adjacency on the positions of the elements $\sigma$, while the set $Y$, symmetrically, identifies constraints on their values. An example of bivincular pattern is depicted in Figure~\ref{figure_bivincular_pattern_132_0_2}.

Given two permutations $\alpha=\alpha_1\cdots\alpha_n$ and $\beta=\beta_1\cdots\beta_m$, the \emph{direct sum} of $\alpha$ and $\beta$ is the permutation $\alpha\oplus\beta=\alpha\beta'$, where $\beta'$ is obtained from $\beta$ by adding $n$ to each of its entries. In other words, $\alpha\oplus\beta$ is the only permutation $\gamma_1\cdots\gamma_n\gamma_{n+1}\cdots\gamma_{n+m}$ such that $\gamma_i<\gamma_j$ for each $i\le n$ and $j\ge n+1$, $\gamma_1\cdots\gamma_n\simeq\alpha$ and $\gamma_{n+1}\cdots\gamma_{n+m}\simeq\beta$. The \emph{skew sum} $\alpha\ominus\beta$ of $\alpha$ and $\beta$ is obtained analogously by requiring $\gamma_i>\gamma_j$ for each $i \le n$ and $j\ge n+1$. Finally, the \emph{reverse} of a permutation $\pi=\pi_1\cdots\pi_n$ is defined as $\pi\reverse=\pi_n\cdots\pi_1$.

\subsection{The $\sigma$-machine}

Let $\sigma$ be a permutation. A \emph{$\sigma$-stack}~\cite{CeClF} is a restricted stack that at all times is not allowed to contain occurrences of $\sigma$, reading its content from top to bottom. The \emph{$\sigma$-machine} consists in making a pass through a $\sigma$-stack, followed by a pass through a $21$-stack, each time using a greedy algorithm. As anticipated in Section~\ref{section_intro}, in this context the word greedy indicates that a push operation is always performed, as long as the constraint of the stack is not violated. Alternatively, we wish to give the following dynamical definition of the $\sigma$-machine. Denote by $\mapsigma{\sigma}$ the $\sigma$-stack operator
\begin{align*}
\mapsigma{\sigma}:\ &\Sym\to\Sym\\
&\pi\mapsto\mapsigma{\sigma}(\pi),
\end{align*}
where $\mapsigma{\sigma}(\pi)$ is the result of the following sorting process on input $\pi$:
\begin{enumerate}
\item Scan $\pi$ from left to right and push the current element onto the $\sigma$-stack, as long as the stack remains $\sigma$-avoding.
\item If pushing the next element of $\pi$ would create an occurrence of $\sigma$ in the $\sigma$-stack, instead pop the top of the $\sigma$-stack into the output.
\item If there are no more elements to be processed, empty the $\sigma$-stack by repeatedly performing a pop operation.
\end{enumerate}
Then, to sort a permutation $\pi$ with the $\sigma$-machine means to compute
$$
\left(\mapsigma{21}\circ\mapsigma{\sigma}\right)(\pi).
$$
As an easy example to familiarize with the $\sigma$-stack, note that if $\pi$ avoids~$\sigma\reverse$, then $\mapsigma{\sigma}\bigl(\pi\bigl)=\pi\reverse$. Indeed, in this case the restriction of the $\sigma$-stack is never triggered and all the entries of $\pi$ are free to enter the $\sigma$-stack. Finally, the $\sigma$-stack is emptied and the output is the reverse of $\pi$. The $\sigma$-machine is illustrated in Figure~\ref{figure_sigma_machine} and the step-by-step action of the $231$-stack on input $2413$ is illustrated in Figure~\ref{figure_sorting_operations}. Let $\identity_n=12\cdots n$ be the increasing permutation of length $n$. A permutation $\pi\in\Sym_n$ is \emph{$\sigma$-sortable} if
$$
\left(\mapsigma{21}\circ\mapsigma{\sigma}\right)(\pi)=\identity_n.
$$

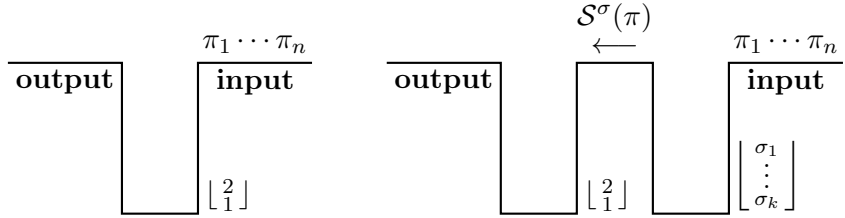
\begin{figure}
\centering
\begin{tikzpicture}[baseline=20pt]
\draw[thick] (0,2)--(1.5,2)--(1.5,0)--(2.5,0)--(2.5,2)--(4,2);
\node at (3.25,1.75){\textbf{input}};
\node at (0.75,1.75){\textbf{output}};
\node at (3.25,2.25){$\pi_1\cdots\pi_n$};
\node at (2.9,0.25){$\left\lfloor\begin{smallmatrix}2\\1\end{smallmatrix}\right\rfloor$};
\end{tikzpicture}
\qquad
\begin{tikzpicture}[baseline=20pt]
\draw[thick] (0,2)--(1.5,2)--(1.5,0)--(2.5,0)--(2.5,2)--(3.5,2)--(3.5,0)--(4.5,0)--(4.5,2)--(6,2);
\node at (5.25,1.75){\textbf{input}};
\node at (0.75,1.75){\textbf{output}};
\node at (5.25,2.25){$\pi_1\cdots\pi_n$};
\node at (2.9,0.25){$\left\lfloor\begin{smallmatrix}2\\1\end{smallmatrix}\right\rfloor$};
\node at (5,0.5){$\left\lfloor\begin{smallmatrix}\sigma_1\\[-1ex] \vdots\\ \sigma_k\end{smallmatrix}\right\rfloor$};
\node[above] at (3,2){$\longleftarrow$};
\node at (3,2.6){$\mapsigma{\sigma}(\pi)$};
\end{tikzpicture}
\caption{The usual representation of sorting with one stack, on the left. The $\sigma$-machine, on the right. The smaller stack in the lower right corner pinpoints the forbidden pattern.}\label{figure_sigma_machine}
\end{figure}

Let $\Sort(\sigma)$ be the set of $\sigma$-sortable permutations and let $\Sort_n(\sigma)=\Sort(\sigma)\cap\Sym_n$. It is well known~\cite{Kn} that a permutation is sortable by a classical stack (i.e. a $21$-stack) if and only if it avoids $231$. Since $\mapsigma{\sigma}(\pi)$ is the input of the $21$-stack, we have that
$$
\pi\in\Sort(\sigma)\text{ if and only if }\mapsigma{\sigma}(\pi)\in\Sym(231)
$$
and thus
$$
\Sort(\sigma)=\bigl(\mapsigma{\sigma}\bigr)^{-1}\bigl(\Sym(231)\bigr).
$$
This simple fact allows us to determine the $\sigma$-sortability of $\pi$ just by checking whether $\mapsigma{\sigma}(\pi)$ avoids or contains $231$. Two general questions that motivate the research on $\sigma$-machines are the following:

\begin{itemize}
\item Characterize the set $\Sort(\sigma)$. Can $\Sort(\sigma)$ be described in terms of (generalized) pattern avoidance?
\item Compute the sequence $\lbrace|\Sort_n(\sigma)|\rbrace_{n\ge 1}$ counting $\sigma$-sortable permutations.
\end{itemize}

The leading paper~\cite{CeClF} provides a characterization of the set of permutations~$\sigma$ such that $\Sort(\sigma)$ is a permutation class. Regarding permutations of length two, a simple inductive argument~\cite{CeClF} shows that $\Sort(12)=\Sym(213)$. On the other hand, $\Sort(21)$ is the set of West-$2$-stack-sortable permutations~\cite{We2}, which is not a class (although it is described by avoidance of a generalized pattern). For what concerns longer permutations, let us recall two more results from~\cite{CeClF}. Given $\sigma\in\Sym_k$, denote by $\hat{\sigma}$ the permutation $\hat{\sigma}=\sigma_2\sigma_1\sigma_3\cdots\sigma_k$ obtained from $\sigma$ by interchanging $\sigma_1$ with $\sigma_2$.

\begin{lemma}\label{lemma_hat_sigma}
Let $\pi$ be an input permutation for the $\sigma$-machine. If $\pi\ge\sigma\reverse$, then $\mapsigma{\sigma}(\pi)\ge\hat{\sigma}$. Otherwise, if $\pi$ avoids $\sigma\reverse$, then $\mapsigma{\sigma}(\pi)=\pi\reverse$. Furthermore, we have:
$$
\Sym\bigl(132,\sigma\reverse\bigr)\subseteq\Sort(\sigma).
$$
\end{lemma}

\begin{theorem}\label{theorem_suff_cond_class}
Let $\sigma$ be a permutation of length three or more. Then $\Sort(\sigma)$ is a class if and only if $\hat{\sigma}\ge 231$. In this case, we have $\Sort(\sigma)=\Sym(132)$, if $\sigma\ge 231$, and $\Sort(\sigma)=\Sym(132,\sigma\reverse)$, if $\sigma\not\ge 231$.
\end{theorem}

Theorem~\ref{theorem_suff_cond_class} immediately solves the problems of characterizing and enumerating the set $\Sort(\sigma)$, when $\Sort(\sigma)$ is a permutation class. For instance, the only class among permutations of length three is $\Sort(321)=\Sym(132,123)$. On the other hand, all the other cases are remarkably arduous. At present, only two have been enumerated: $\Sort(123)$, by means of a bijection with a set of pattern-avoiding Schr\"oder paths~\cite{CeClF}; and $\Sort(132)$, by means of a bijection with a class of pattern-avoiding restricted growth functions~\cite{CeClFS}.

\begin{figure}
\scalebox{0.925}{
$$
\def\arraystretch{4.5}
\begin{tabular}{c|c|c|c}
\begin{tikzpicture}[scale=1, baseline=20pt]
\draw[thick] (0,1.5)--(1,1.5)--(1,0)--(1.5,0)--(1.5,1.5)--(2.5,1.5);
\node at (2,1.25){\footnotesize{input}};
\node at (0.45,1.25){\footnotesize{output}};
\node at (0.25,0.25){\footnotesize{Step 1}};
\node at (1.775,0.25){\scriptsize{$\left\lfloor\begin{smallmatrix}2\\3\\1\end{smallmatrix}\right\rfloor$}};
\node at (2,1.75){$\cancel{2}413$};
\node at (1.25,0.2){$2$};
\end{tikzpicture}
&
\begin{tikzpicture}[scale=1, baseline=20pt]
\draw[thick] (0,1.5)--(1,1.5)--(1,0)--(1.5,0)--(1.5,1.5)--(2.5,1.5);
\node at (2,1.25){\footnotesize{input}};
\node at (0.45,1.25){\footnotesize{output}};
\node at (0.25,0.25){\footnotesize{Step 2}};
\node at (1.775,0.25){\scriptsize{$\left\lfloor\begin{smallmatrix}2\\3\\1\end{smallmatrix}\right\rfloor$}};
\node at (2,1.75){$\cancel{4}13$};
\node at (1.25,0.2){$2$};
\node at (1.25,0.6){$4$};
\end{tikzpicture}
&
\begin{tikzpicture}[scale=1, baseline=20pt]
\draw[thick] (0,1.5)--(1,1.5)--(1,0)--(1.5,0)--(1.5,1.5)--(2.5,1.5);
\node at (2,1.25){\footnotesize{input}};
\node at (0.45,1.25){\footnotesize{output}};
\node at (0.25,0.25){\footnotesize{Step 3 }};
\node at (1.775,0.25){\scriptsize{$\left\lfloor\begin{smallmatrix}2\\3\\1\end{smallmatrix}\right\rfloor$}};
\node at (2,1.75){$\cancel{1}3$};
\node at (1.25,0.25){$2$};
\node at (1.25,0.6){$4$};
\node at (1.25,1){$1$};
\end{tikzpicture}
&
\begin{tikzpicture}[scale=1, baseline=20pt]
\draw[thick] (0,1.5)--(1,1.5)--(1,0)--(1.5,0)--(1.5,1.5)--(2.5,1.5);
\node at (2,1.25){\footnotesize{input}};
\node at (0.45,1.25){\footnotesize{output}};
\node at (0.25,0.25){\footnotesize{Step 4}};
\node at (1.775,0.25){\scriptsize{$\left\lfloor\begin{smallmatrix}2\\3\\1\end{smallmatrix}\right\rfloor$}};
\node at (2,1.75){$3$};
\node at (1.25,0.25){$2$};
\node at (1.25,0.6){$4$};
\node at (1.25,1){$\cancel{1}$};
\node at (0.5,1.75){$1$};
\end{tikzpicture}\\
\hline
\begin{tikzpicture}[scale=1, baseline=20pt]
\draw[thick] (0,1.5)--(1,1.5)--(1,0)--(1.5,0)--(1.5,1.5)--(2.5,1.5);
\node at (2,1.25){\footnotesize{input}};
\node at (0.45,1.25){\footnotesize{output}};
\node at (0.25,0.25){\footnotesize{Step 5}};
\node at (1.775,0.25){\scriptsize{$\left\lfloor\begin{smallmatrix}2\\3\\1\end{smallmatrix}\right\rfloor$}};
\node at (2,1.75){$3$};
\node at (1.25,0.25){$2$};
\node at (1.25,0.6){$\cancel{4}$};
\node at (0.5,1.75){$14$};
\end{tikzpicture}
&
\begin{tikzpicture}[scale=1, baseline=20pt]
\draw[thick] (0,1.5)--(1,1.5)--(1,0)--(1.5,0)--(1.5,1.5)--(2.5,1.5);
\node at (2,1.25){\footnotesize{input}};
\node at (0.45,1.25){\footnotesize{output}};
\node at (0.25,0.25){\footnotesize{Step 6}};
\node at (1.775,0.25){\scriptsize{$\left\lfloor\begin{smallmatrix}2\\3\\1\end{smallmatrix}\right\rfloor$}};
\node at (2,1.75){$\cancel{3}$};
\node at (1.25,0.25){$2$};
\node at (1.25,0.6){$3$};
\node at (0.5,1.75){$14$};
\end{tikzpicture}
&
\begin{tikzpicture}[scale=1, baseline=20pt]
\draw[thick] (0,1.5)--(1,1.5)--(1,0)--(1.5,0)--(1.5,1.5)--(2.5,1.5);
\node at (2,1.25){\footnotesize{input}};
\node at (0.45,1.25){\footnotesize{output}};
\node at (0.25,0.25){\footnotesize{Step 7}};
\node at (1.775,0.25){\scriptsize{$\left\lfloor\begin{smallmatrix}2\\3\\1\end{smallmatrix}\right\rfloor$}};
\node at (1.25,0.25){$2$};
\node at (1.25,0.6){$\cancel{3}$};
\node at (0.5,1.75){$143$};
\end{tikzpicture}
&
\begin{tikzpicture}[scale=1, baseline=20pt]
\draw[thick] (0,1.5)--(1,1.5)--(1,0)--(1.5,0)--(1.5,1.5)--(2.5,1.5);
\node at (2,1.25){\footnotesize{input}};
\node at (0.45,1.25){\footnotesize{output}};
\node at (0.25,0.25){\footnotesize{Step 8}};
\node at (1.775,0.25){\scriptsize{$\left\lfloor\begin{smallmatrix}2\\3\\1\end{smallmatrix}\right\rfloor$}};
\node at (1.25,0.25){$\cancel{2}$};
\node at (0.5,1.75){$1432$};
\end{tikzpicture}
\\
\end{tabular}
$$
}
\caption{Step-by-step computation of $\mapsigma{231}(2413)=1432$.}\label{figure_sorting_operations}
\end{figure}
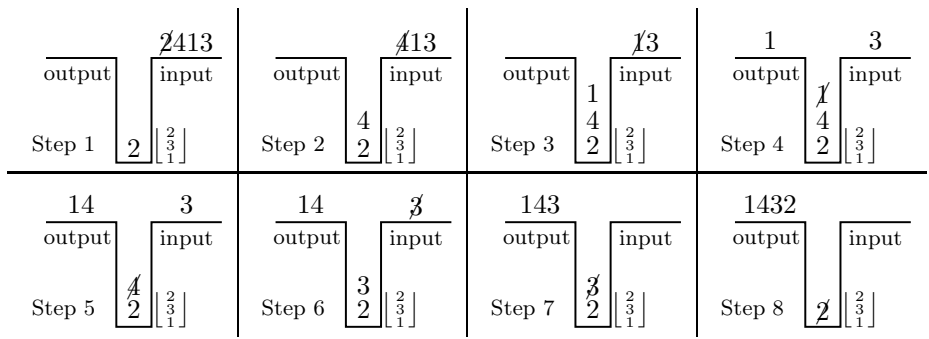

\section{A family of challenging $\sigma$-machines}\label{section_bivincular_result}

Let $\xi=(132,\lbrace 0,2\rbrace,\emptyset)$ be the bivincular pattern depicted in Figure~\ref{figure_bivincular_pattern_132_0_2}. Recall that an occurrence of $\xi$ in a permutation $\pi$ is a classical occurrence $\pi_i\pi_j\pi_k$ of $132$ such that $i=1$ and $k=j+1$. Alternatively, we have:
$$
\Sym(\xi)=\lbrace \pi\in\Sym: \text{if $\pi_i\pi_j\pi_k\simeq 132$, then $i\neq 1$ or $k\neq j+1$}\rbrace.
$$
The goal of this section is to show that $\Sort(\sigma)$ is a subset of $\Sym(\xi)$, unless $\sigma$ is the skew sum of $12$ with a $231$-avoiding permutation. As we will show in what follows, the geometric structure of $\Sym(\xi)$ is fairly straightforward and transparent. This suggests that the cases where $\Sort(\sigma)$ is not contained in $\Sym(\xi)$ could be the most challenging to be studied. The shortest such permutation is~$231$. In fact, as suggested by some data, the~$231$-machine seems to be the~$\sigma$-machine that sorts the largest amount of permutations. For example, it is the only one that can sort every permutation of length three.

We start by providing a simple decomposition of permutations avoiding $\xi$, from which the enumeration of $\Sym(\xi)$ follows easily. Let $\pi\in\Sym$ and suppose that $\pi_1=t+1$, for some $t\ge 0$. Write
$$
\pi=\pi_1 B_0 b_1 B_1 b_2 B_2\cdots b_t B_t,
$$
where $\lbrace b_1,\dots,b_t\rbrace=\lbrace 1,\dots,t\rbrace$, $B_0$ contains the entries between $\pi_1$ and $b_1$, and $B_i$ contains the entries between $b_i$ and $b_{i+1}$, for $i\ge 1$. We refer to this as the \emph{first-element decomposition} of $\pi$, and $B_i$ is said to be the $i$-th \emph{block}. The first-element decomposition of a permutation $\pi\in\Sym(\xi)$ is illustrated in Figure~\ref{figure_bivincular_pattern_132_0_2}.

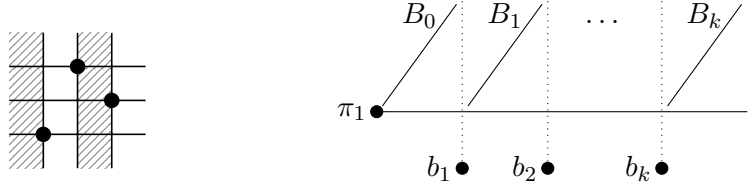
\begin{figure}[h!]
\begin{center}
\begin{DrawPerm}
\meshBox{(0,0)}{(1,4)}
\meshBox{(2,0)}{(3,4)}
\fillPerm{1,3,2}{3.99}{3.99}
\end{DrawPerm}
\hspace{2cm}
\begin{tikzpicture}[scale=0.75, baseline=20pt]
\draw (0,1)--(6.5,1);
\draw (0.1,1.1)--(1.4,2.9);
\draw (1.6,1.1)--(2.9,2.9);
\draw (5.1,1.1)--(6.4,2.9);
\draw[dotted] (1.5,0.01)--(1.5,2.99);
\draw[dotted] (3,0.01)--(3,2.99);
\draw[dotted] (5,0.01)--(5,2.99);
\filldraw (0,1) circle (3pt) node[left]{$\pi_1$};
\filldraw (1.5,0) circle (3pt) node[left]{$b_1$};
\filldraw (3,0) circle (3pt) node[left]{$b_2$};
\filldraw (5,0) circle (3pt) node[left]{$b_k$};
\node at(0.75,2.7){$B_0$};
\node at(2.25,2.7){$B_1$};
\node at(4,2.6){$\dots$};
\node at(5.75,2.7){$B_k$};
\end{tikzpicture}
\end{center}
\caption{On the left, the bivincular pattern $\xi=(132,\lbrace 0,2\rbrace,\emptyset)$. A permutation $\pi$ contains $\xi$ if $\pi$ contains an occurrence of $132$ such that no other point falls in a shaded region. On the right, the first-element decomposition of $\pi\in\Sym(\xi)$.}\label{figure_bivincular_pattern_132_0_2}
\end{figure}

\begin{lemma}\label{lemma_first_el_dec_increasing_blocks}
A permutation $\pi$ avoids $\xi$ if and only if each block in the first-element decomposition of $\pi$ is increasing.
\end{lemma}
\begin{proof}
Let $\pi=\pi_1 B_0 b_1 B_1 b_2 B_2\cdots b_t B_t$ be the first-element decomposition of~$\pi$ and let $i\ge 0$. Since each element of $B_i$ is greater than $\pi_1$, any descent in $B_i$ would result in an occurrence of $\xi$ in $\pi$. Thus $B_i$ is increasing if $\pi$ avoids $\xi$. On the other hand, let $\pi_1\pi_j\pi_{j+1}$ be an occurrence of $\xi$, for some $j\ge 2$. Note that $\pi_j>\pi_{j+1}>\pi_1$. Thus $\pi_j$ and $\pi_{j+1}$ are in the same block, which is not increasing.
\end{proof}

\begin{corollary}\label{lemma_first_el_dec_identity}
Let $\pi\in\Sym_n(\xi)$. If $\pi_1=1$, then $\pi=\identity_n$.
\end{corollary}
\begin{proof}
If $\pi_1=1$, then the first-element decomposition of $\pi$ contains only one block, which is increasing due to Lemma~\ref{lemma_first_el_dec_increasing_blocks}.
\end{proof}

\begin{theorem}\label{theorem_first_el_dec_enumer}
For $n\ge 1$, we have:
$$
|\Sym_n(\xi)|=\displaystyle{\sum_{t=0}^{n-1} t! (t+1)^{n-t-1}}
\qquad\text{(sequence $A129591$ in~\cite{Sl}).}
$$
\end{theorem}
\begin{proof}
For $t=0,1,\dots,n-1$, let
$$
\Sym_n^t(\xi)=\lbrace\pi\in\Sym_n(\xi):\pi_1=t+1\rbrace
\quad\text{and}\quad
f_{n,t}=|\Sym_n^t(\xi)|.
$$
We show that $f_{n,t}=t!(t+1)^{n-t-1}$, from which the thesis follows immediately. Any permutation $\pi\in\Sym_n^t(\xi)$ can be constructed uniquely as follows. The entries of $\pi$ that are smaller than $\pi_1=t+1$ can be chosen freely, since they cannot contribute to an occurrence of $\xi$. This can be done in $t!$ distinct ways. By Lemma~\ref{lemma_first_el_dec_increasing_blocks}, entries that are greater than $\pi_1$ must be arranged in increasing blocks. In other words, for each such entry it is sufficient to choose the index $j\in\lbrace 0,1,\dots,t\rbrace$ of the (increasing) block it belongs to. That is, there are $t+1$ choices for each of the $n-t-1$ elements that are greater than $\pi_1$. Therefore $f_{n,t}=t!\cdot(t+1)^{n-t-1}$, as desired.
\end{proof}

Next we wish to prove the main result of this section.

\begin{theorem}\label{theorem_bivincular_pattern_gen_result}
Let $\sigma=\sigma_1\cdots\sigma_k$ be a permutation of length at least three. The following three conditions are equivalent.
\begin{enumerate}
\item[$(1)$] $\Sort(\sigma)\not\subseteq\Sym(\xi)$.
\item[$(2)$] $\sigma= 12\ominus\beta$, for some $\beta\in\Sym(231)$.
\item[$(3)$] $\hat{\sigma}\in\Sym(231)$ and $\sigma\reverse\ge\xi$.
\end{enumerate}
\end{theorem}
\begin{proof}
$(1)\hspace{-.2em}\Longrightarrow\hspace{-.2em}(2)$: Suppose that $\Sort(\sigma)\not\subseteq\Sym(\xi)$. We wish to prove that $\sigma_2>\sigma_1$, $\sigma_1>\sigma_i$, for each $i\ge 3$, and $\sigma_3\cdots\sigma_k$ avoids $231$. Let us start by showing that $\hat{\sigma}$ avoids $231$, which implies that $\sigma_3\cdots\sigma_k$ avoids $231$ as well. For a contradiction, suppose that $\hat{\sigma}\ge 231$. By Theorem~\ref{theorem_suff_cond_class}, we have:
$$
\Sort(\sigma)=\Sym\bigl(132,\sigma\reverse\bigr)\subseteq\Sym(\xi),
$$
where the last inequality follows from the fact that $132$ is the classical pattern underlying $\xi$ (and thus $\Sym(132)\subseteq\Sym(\xi)$). Hence it would be $\Sort(\sigma)\subseteq\Sym(\xi)$, which contradicts the hypothesis.
Now, due to the assumption $\Sort(\sigma)\not\subseteq\Sym(\xi)$, there exists a permutation $\pi$ such that $\pi\in\Sort(\sigma)$ and $\pi\ge\xi$. Let $\pi_1\pi_j\pi_{j+1}$ be an occurrence of $\xi$ in $\pi$. Consider the action of the $\sigma$-stack on input $\pi$. Since $\sigma$ has length at least three, $\pi_1$ remains at the bottom of the $\sigma$-stack until the end of the sorting process; that is, the last element that exits the $\sigma$-stack is $\pi_1$. Therefore $\pi_j$ must be extracted before $\pi_{j+1}$ enters, or else $\pi_{j+1}\pi_j\pi_1$ would be an occurrence of $231$ in $\mapsigma{\sigma}(\pi)$, contradicting the fact that $\pi$ is $\sigma$-sortable. Let us consider the instant when $\pi_j$ is extracted, as illustrated below:

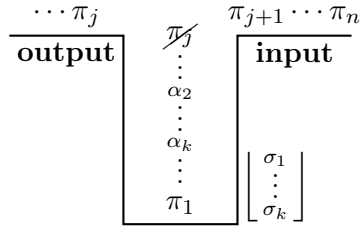
\begin{center}
\begin{tikzpicture}[baseline=20pt]
\draw[thick] (0,2.5)--(1.5,2.5)--(1.5,0)--(3,0)--(3,2.5)--(4.5,2.5);
\node at (3.75,2.25){\textbf{input}};
\node at (0.75,2.25){\textbf{output}};
\node at (3.75,2.75){$\pi_{j+1}\cdots\pi_n$};
\node at (0.75,2.75){$\cdots\pi_j$};
\node at (3.5,0.5){$\left\lfloor\begin{smallmatrix}\sigma_1\\[-1ex] \vdots\\ \sigma_k\end{smallmatrix}\right\rfloor$};
\node at (2.25,2.5){$\cancel{\pi_j}$};
\node at (2.25,1.5){$\begin{smallmatrix}\vdots\\ \alpha_2\\[-1ex]\vdots\\ \alpha_k\\[-1ex]\vdots\end{smallmatrix}$};
\node at (2.25,0.25){$\pi_1$};
\end{tikzpicture}
\end{center}

Note that $\pi_{j+1}$ is the next element of the input. Since a pop operation is performed, the $\sigma$-stack contains a subsequence of $k-1$ elements, say $\alpha_2\alpha_3\cdots\alpha_k$ (reading from top to bottom), such that:
$$
\pi_{j+1}\alpha_2\alpha_3\cdots\alpha_k\simeq\sigma,
$$
where it could be $\alpha_2=\pi_j$ or $\alpha_k=\pi_1$. In particular, all the elements above~$\alpha_2$, including $\alpha_2$, are extracted before $\pi_{j+1}$ enters. Without losing generality, we can assume that $\alpha_3$ is still contained in the $\sigma$-stack when $\pi_{j+1}$ enters; this can be achieved, for instance, by taking the ``deepest'' such sequence $\alpha_2\cdots\alpha_k$ contained in the $\sigma$-stack. As a result, $\mapsigma{\sigma}(\pi)$ contains an occurrence $\alpha_2\pi_{j+1}\alpha_3\cdots\alpha_k$ of $\hat{\sigma}$. Now, our goal is to show that $\sigma_2>\sigma_1$ and $\sigma_1>\sigma_i$, for each $i\ge 3$; equivalently, we shall prove that:
$$
\text{(i)}\ \pi_{j+1}>\alpha_i, \text{ for each }i\ge 3,
\quad\text{and}\quad
\text{(ii)}\ \alpha_2>\pi_{j+1}.
$$
\begin{itemize}
\item[(i)] For a contradiction, suppose that $\pi_{j+1}<\alpha_i$, for some $i\ge 3$. Then $\alpha_i\neq\pi_1$ (because $\pi_1<\pi_{j+1}<\alpha_i$) and $\pi_{j+1}\alpha_i\pi_1$ is an occurrence of $231$ in $\mapsigma{\sigma}(\pi)$, a contradiction with $\pi$ being $\sigma$-sortable.
\item[(ii)] For a contradiction, suppose that $\alpha_2<\pi_{j+1}$. Then, due to the previous item and since $\pi_j>\pi_{j+1}$, we have that $\pi_j\alpha_2\alpha_3\cdots\alpha_k\simeq\sigma$. But this is impossible since the $\sigma$-stack contains $\pi_j\alpha_2\alpha_3\cdots\alpha_k$ when $\pi_j$ is extracted.
\end{itemize} 

$(2)\hspace{-0.35em}\Longrightarrow\hspace{-0.35em}(3)$: Suppose that $\sigma= 12\ominus\beta=(k-1)k\beta$, for some $\beta=\beta_1\cdots\beta_{k-2}\in\Sym(231)$. Then $\hat{\sigma}=k(k-1)\beta$ avoids $231$, since $\beta$ does so. Finally, we have
$$
\sigma\reverse=\beta_{k-2}\cdots\beta_1k(k-1)
$$
and $\beta_{k-2}k(k-1)$ is an occurrence of $\xi$ in $\sigma\reverse$.

$(3)\hspace{-.2em}\Longrightarrow\hspace{-.2em}(1)$: Suppose that $\hat{\sigma}$ avoids $231$ and $\sigma\reverse\ge\xi$. We show that $\sigma\reverse\in\Sort(\sigma)$. Due to Lemma~\ref{lemma_hat_sigma}, we have $\mapsigma{\sigma}\bigl(\sigma\reverse\bigr)=\hat{\sigma}$. Finally, $\hat{\sigma}$ avoids $231$, hence $\sigma\reverse$ is $\sigma$-sortable. We have thus proved that $\sigma\reverse\in\Sort(\sigma)$ and $\sigma\reverse\notin\Sym(\xi)$, as desired.
\end{proof}

Due to Corollary~\ref{lemma_first_el_dec_identity}, if $\Sort(\sigma)\subseteq\Sym(\xi)$, then the only $\sigma$-sortable permutation that starts with $1$ is the identity permutation. This is not true in general. For instance, the permutations $12354$, $12453$, $12534$ and $12543$ are $3421$-sortable.


\section{$\sigma$-sorted and effective permutations}\label{section_sorted_perm_fertilities}

The previous works related to $\sigma$-machines, including Section~\ref{section_bivincular_result} of this paper, have been devoted to the investigation of $\sigma$-sortable permutations, that is, the preimage of $\Sym(231)$ via $\mapsigma{\sigma}$. In this section we adopt a slightly different approach and consider the operator $\mapsigma{\sigma}$ from a dynamical perspective.

Let $\sigma$ be a permutation. The set of \emph{$\sigma$-sorted permutations} is defined as
$$
\sorted(\sigma)=\mapsigma{\sigma}(\Sort(\sigma)).
$$
Equivalently, permutations in $\sorted(\sigma)$ are those outputs of the $\sigma$-stack that are also $12$-sortable, that is:
$$
\sorted(\sigma)=\mapsigma{\sigma}(\Sym)\cap\Sym(231).
$$
We remark that a different notion of sorted permutations could be obtained by considering the set $\mapsigma{\sigma}(\Sym)$ of all the possible outputs of the $\sigma$-stack (including those that are not $12$-sortable). This framework was adopted recently~\cite{Ce} in the more general case of Cayley permutations.

The \emph{$\sigma$-fertility} of a permutation~$\gamma$ is $\fert{\sigma}(\gamma)=|\bigl(\mapsigma{\sigma}\bigr)^{-1}(\gamma)|$; less formally, $\fert{\sigma}(\gamma)$ equals the number of ways $\gamma$ can be obtained as an output of the $\sigma$-stack. The word ``fertility'' originally appeared in West's dissertation~\cite{We} with regard to the classical $21$-machine. More recently, fertility under various settings has been investigated by Defant {\etal}~\cite{De,De2,DEM,DZ}. Note that:

\begin{equation}\label{eqn:eq_sorted}
|\Sort_n(\sigma)|=\sum_{\gamma\in\sorted(\sigma)}\fert{\sigma}(\gamma).
\end{equation}

A permutation $\sigma$ is said to be \emph{effective} if $\sorted(\sigma)\subseteq\Sym(\sigma)$. Alternatively, if the $\sigma$-stack succesfully performs its intended task of preventing occurrences of $\sigma$ to be produced, at least when the input is $\sigma$-sortable. For instance, the permutation $21$ is not effective. Indeed we have $\mapsigma{21}(231)=213$ (which avoids $231$). Hence $231$ is $21$-sortable, but $\mapsigma{21}(231)$ contains $21$. On the other hand, $231$ is trivially effective since $\sorted(231)\subseteq\Sym(231)$ by definition of sorted permutations.

The motivation to study the notion of effectiveness lies in the fact that, as we show in Proposition~\ref{proposition_effective_patterns}, a permutation $\sigma$ is effective if and only if $\sorted(\sigma)=\Sym(231,\sigma)$. In this case, we can rewrite Equation~\ref{eqn:eq_sorted} as:
$$
|\Sort_n(\sigma)|=\sum_{\gamma\in\Sym(231,\sigma)}\fert{\sigma}(\gamma),
$$
thus obtaining an alternative method to enumerate $\Sort(\sigma)$. An instance where this approach is suitable will be illustrated in Example~\ref{example_123_machine}.

\begin{proposition}\label{proposition_effective_patterns}
A permutation $\sigma$ is effective if and only if $\sorted(\sigma)=\Sym(231,\sigma)$.
\end{proposition}
\begin{proof}
If $\sorted(\sigma)=\Sym(231,\sigma)$, then $\sigma$ is effective since $\Sym(231,\sigma)\subseteq\Sym(\sigma)$. To prove the opposite statement, suppose that $\sigma$ is effective, that is, $\sorted(\sigma)\subseteq\Sym(\sigma)$. Note that $\sorted(\sigma)\subseteq\Sym(231)$ by definition, and thus we have $\sorted(\sigma)\subseteq\Sym(231,\sigma)$. We wish to prove the opposite inclusion $\Sym(231,\sigma)\subseteq\sorted(\sigma)$, from which the thesis follows. Recall that $\Sym\bigl(132,\sigma\reverse\bigr)\subseteq\Sort(\sigma)$ by Lemma~\ref{lemma_hat_sigma}. Note that $\mapsigma{\sigma}$ acts as the reverse on any permutation that avoids $\sigma\reverse$, and in particular on $\Sym\bigl(132,\sigma\reverse\bigr)$. Therefore:
$$
\mapsigma{\sigma}\bigl(\Sym(132,\sigma\reverse)\bigr)=
\Sym(231,\sigma),
$$
and thus $\Sym(231,\sigma)\subseteq\sorted(\sigma)$.
\end{proof}

Our next goal is to determine which permutations are effective.

\begin{proposition}\label{theorem_sorted_hat_231}
Let $\sigma$ be a permutation. If $\hat{\sigma}\ge 231$, then $\sigma$ is effective.
\end{proposition}
\begin{proof}
Suppose that $\hat{\sigma}\ge 231$. By Theorem~\ref{theorem_suff_cond_class}, we have $\Sort(\sigma)=\Sym\bigl(132,\sigma\reverse\bigr)$. In particular, as observed in the proof of Proposition~\ref{proposition_effective_patterns}, $\mapsigma{\sigma}$ acts as the reverse operator on $\Sym\bigl(132,\sigma\reverse\bigr)$. Hence $\sorted(\sigma)=\Sym(231,\sigma)$ and $\sigma$ is effective due to Proposition~\ref{proposition_effective_patterns}.
\end{proof}

Due to Theorem~\ref{theorem_suff_cond_class} and Proposition~\ref{theorem_sorted_hat_231}, if both $\hat{\sigma}$ and $\sigma$ contain $231$, then we have:
$$
\Sort(\sigma)=\Sym(132)
\quad\text{and}\quad
\sorted(\sigma)=\Sym(231,\sigma)=\Sym(231),
$$
hence $\mapsigma{\sigma}$ is a bijection between $\Sort(\sigma)$ and $\sorted(\sigma)$. On the other hand, if $\hat{\sigma}$ contains $231$ and $\sigma$ avoids $231$, we have that $\sorted(\sigma)=\Sym(231,\sigma)$ and $\sorted(\sigma)$ is strictly contained in $\Sym(231)$.

\begin{proposition}\label{proposition_sorted_necessary}
Let $\sigma$ be a permutation of length two or more. If $\sigma$ avoids~$231$ and $\sigma_2=1$, then $\sigma$ is not effective.
\end{proposition}
\begin{proof}
To show that $\sigma$ is not effective, we wish to construct a $\sigma$-sortable permutation $\pi$ such that $\mapsigma{\sigma}(\pi)\ge\sigma$. Let $\sigma\in\Sym_k(231)$, with $k\ge 2$ and $\sigma_2=1$. Let $\sigma_1=t$, for some $t\ge 2$. Since $\sigma$ avoids $231$, it must be
$$
\lbrace\sigma_3,\dots,\sigma_t\rbrace=\lbrace 2,\dots,t-1\rbrace\text{ and }\lbrace\sigma_{t+1},\dots,\sigma_k\rbrace=\lbrace t+1,\dots,k\rbrace.
$$
Otherwise, there would be two indices, say $i\in\lbrace 3,\dots,t\rbrace$ and $j\in\lbrace t+1,\dots,k\rbrace$, such that $\sigma_i>t$ and $\sigma_j<t$. But then $\sigma_1\sigma_i\sigma_j$ would be an occurrence of $231$ in $\sigma$ (see also Figure~\ref{figure_sigma_not_effective}). An immediate consequence is that $\sigma_2\cdots\sigma_t$ is a permutation of length $t-1$. Define the permutation $\pi$ as
$$
\pi=\sigma\reverse\ominus(\sigma_2\cdots\sigma_t)\reverse=
\sigma'_k\cdots\sigma'_2\sigma'_1\sigma_t\cdots\sigma_2,
$$
where $\sigma'_i=\sigma_i+t-1$, for each $i$. We wish to prove that:
$$
\mapsigma{\sigma}(\pi)=\sigma'_2\sigma_2\cdots\sigma_t\sigma'_1\sigma'_3\cdots\sigma'_k.
$$
Let us consider the action of the $\sigma$-stack on input $\pi$. The first element of $\pi$ that cannot be pushed in the $\sigma$-stack is $\sigma'_1$, since $\sigma'_1\sigma'_2\cdots\sigma'_k$ is an occurrence of $\sigma$. Thus $\sigma'_2$ is extracted when $\sigma'_1$ is the next element of the input, and then $\sigma'_1$ is pushed on top, as illustrated in Figure~\ref{figure_sigma_not_effective}. We claim that the remaining elements $\sigma_t\cdots\sigma_2$ are free to enter the $\sigma$-stack; equivalently, that
$$
\sigma_2\cdots\sigma_t\sigma'_1\sigma'_3\cdots\sigma'_k
$$
does not contain $\sigma$. Indeed, we have $\sigma_i<t$ for each $i\in\lbrace 2,\dots,t\rbrace$, whereas $\sigma_1=t$; thus no element in the prefix $\sigma_2\cdots\sigma_t$ can play the role of $\sigma_1$ in an occurrence of $\sigma$. Furthermore, the suffix $\sigma'_1\sigma'_3\cdots\sigma'_k$ is too short to contain~$\sigma$ (it has length one less than $\sigma$). As a result, all the remaining elements $\sigma_t\cdots\sigma_2$ are pushed in the $\sigma$-stack above $\sigma'_1\sigma'_3\cdots\sigma'_k$. At the end, we obtain:
$$
\mapsigma{\sigma}(\pi)=
\sigma'_2\sigma_2\cdots\sigma_t\sigma'_1\sigma'_3\cdots\sigma'_k=
\sigma_1\sigma_2\cdots\sigma_t\sigma'_1\sigma'_3\cdots\sigma_k',
$$
where the last equality is due to the fact that:
$$
\sigma'_2=\sigma_2+t-1=1+t-1=t=\sigma_1.
$$
To complete the proof, we need to show that $\pi$ is $\sigma$-sortable and $\mapsigma{\sigma}(\pi)\ge\sigma$. To prove that $\pi$ is $\sigma$-sortable, one can easily check that $\mapsigma{\sigma}(\pi)$ avoids $231$. Indeed $\sigma$ avoids $231$; and there are no occurrences of $231$ across $\sigma_1\sigma_2\cdots\sigma_t$ and $\sigma'_1\sigma'_3\cdots\sigma_k'$, since $\sigma_i<\sigma'_j$ for each $i\in\lbrace 1,\dots,t\rbrace$ and $j\in\lbrace 1,3,\dots,k\rbrace$. Finally, $\mapsigma{\sigma}(\pi)$ contains the occurrence $\sigma_1\sigma_2\cdots\sigma_t\sigma'_{t+1}\cdots\sigma'_k$ of $\sigma$.
\end{proof}

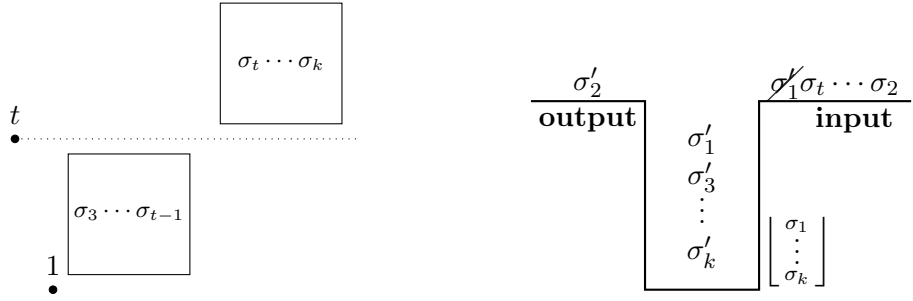
\begin{figure}
\begin{center}
\begin{tikzpicture}[baseline=20pt]
\node[circle, draw, fill=black, inner sep=1pt, label=$t$]  at (0,2){};
\node[circle, draw, fill=black, inner sep=1pt, label=$1$]  at (0.5,0){};
\draw[dotted] (0,2)--(4.5,2);
\draw (0.7,0.2) rectangle (2.3,1.8);
\draw (2.7,2.2) rectangle (4.3,3.8);
\node at (1.5,1){\footnotesize{$\sigma_3\cdots\sigma_{t-1}$}};
\node at (3.5,3){\footnotesize{$\sigma_t\cdots\sigma_k$}};
\end{tikzpicture}
\hspace{2cm}
\begin{tikzpicture}[baseline=20pt]
\draw[thick] (0,2.5)--(1.5,2.5)--(1.5,0)--(3,0)--(3,2.5)--(5,2.5);
\node at (4.25,2.25){\textbf{input}};
\node at (0.75,2.25){\textbf{output}};
\node at (3.5,0.5){$\left\lfloor\begin{smallmatrix}\sigma_1\\[-1ex] \vdots\\ \sigma_k\end{smallmatrix}\right\rfloor$};
\node at (4,2.75){$\cancel{\sigma'_1}\sigma_t\cdots\sigma_2$};
\node at (0.75,2.75){$\sigma'_2$};
\node at (2.25,2){$\sigma'_1$};
\node at (2.25,1){$\begin{matrix}\sigma'_3\\[-1ex] \vdots\\ \sigma'_k\end{matrix}$};
\end{tikzpicture}
\end{center}
\caption{Referring to Proposition~\ref{proposition_sorted_necessary}, the permutation $\sigma$, on the left, and the action of the $\sigma$-stack on $\pi$, on the right.}\label{figure_sigma_not_effective}
\end{figure}

It is easy to see that the hypotheses of Proposition~\ref{proposition_sorted_necessary}:
$$
\sigma\in\Sym(231)\quad\text{and}\quad\sigma_2=1
$$
can be reformulated equivalently as
$$
\hat{\sigma}=1\oplus\alpha,\text{ with }\alpha\in\Sym(231).
$$
Next we wish to prove that if $\hat{\sigma}$ is not the direct sum of $1$ plus a $231$-avoiding permutation, then $\sigma$ is effective. If $\hat{\sigma}\ge 231$, then $\sigma$ is effective by Proposition~\ref{theorem_sorted_hat_231}. Below we address the remaining case where $\hat{\sigma}$ avoids $231$ and $\hat{\sigma}_1\neq 1$.

\begin{proposition}\label{proposition_sorted_suff}
Let $\sigma$ be a permutation. If $\hat{\sigma}$ avoids $231$ and $\hat{\sigma}_1\neq 1$, then $\sigma$ is effective.
\end{proposition}
\begin{proof}
We wish to prove that $\sorted(\sigma)\subseteq\Sym(\sigma)$. Let $\pi$ be a $\sigma$-sortable permutation and let $\mapsigma{\sigma}(\pi)=\gamma$. We show that $\gamma$ avoids $\sigma$. If $\pi$ avoids $\sigma\reverse$, then $\gamma=\pi\reverse$ and $\gamma$ avoids $\sigma$. Otherwise, suppose that $\pi\ge\sigma\reverse$. By Lemma~\ref{lemma_hat_sigma}, we have that $\gamma\ge\hat{\sigma}$. In particular, $\gamma$ contains both $\sigma$ and $\hat{\sigma}$; and $\gamma$ avoids $231$ since $\gamma\in\sorted(\sigma)$.
Note that, as we assumed $\hat{\sigma}_2\neq 1$, i.e. $\sigma_1\neq 1$, it must necessarily be $\sigma_2=1$. Otherwise, if $\sigma_i=1$ for some $i\ge 3$, we would have either $\sigma_1\sigma_2\sigma_i\simeq 231$, if $\sigma_1<\sigma_2$, or $\sigma_2\sigma_1\sigma_i \simeq 231$, if $\sigma_1>\sigma_2$. In the former case, $\sigma$ would contain $231$, which is impossible since $\gamma$ contains $\sigma$ and avoids $231$. In the latter case, analogously, $\hat{\sigma}$ would contain $231$, which is impossible since $\gamma$ contains $\hat{\sigma}$ and avoids $231$.

Now, let $\tilde{\sigma}_1\cdots\tilde{\sigma}_k$ be an occurrenceof $\sigma$ in $\gamma$. Let us consider the instant when $\tilde{\sigma}_1$ is extracted from the $\sigma$-stack (see also the diagram below). At this point, at least one entry $\tilde{\sigma}_j$, with $j\ge 2$, must be in the remaining part of the input. Otherwise, the $\sigma$-stack would contain all the entries $\tilde{\sigma}_1\cdots\tilde{\sigma}_k\simeq\sigma$, which is forbidden. Therefore, the reason why $\tilde{\sigma}_1$ is extracted is because the next element of the input triggers the restriction of the $\sigma$-stack (and not because there are no more elements to be processed). Equivalently, the next element of the input, say $\sigma'_1$, realizes an occurrence of $\sigma$ with some elements $\sigma'_2\cdots\sigma'_k$ contained in the $\sigma$-stack (from top to bottom). This is illustrated by the following diagram:

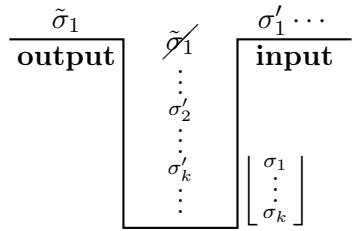
\begin{center}
\begin{tikzpicture}[baseline=20pt]
\draw[thick] (0,2.5)--(1.5,2.5)--(1.5,0)--(3,0)--(3,2.5)--(4.5,2.5);
\node at (3.75,2.25){\textbf{input}};
\node at (0.75,2.25){\textbf{output}};
\node at (3.75,2.75){$\sigma'_1\cdots$};
\node at (0.75,2.75){$\tilde{\sigma}_1$};
\node at (3.5,0.5){$\left\lfloor\begin{smallmatrix}\sigma_1\\[-1ex] \vdots\\ \sigma_k\end{smallmatrix}\right\rfloor$};
\node at (2.25,2.5){$\cancel{\tilde{\sigma}_1}$};
\node at (2.25,1.25){$\begin{smallmatrix}\vdots\\ \sigma'_2\\[-1ex] \vdots\\ \sigma'_k\\[-1ex] \vdots\end{smallmatrix}$};
\end{tikzpicture}
\end{center}

Note that, since $\sigma_1=1$, it must be $\tilde{\sigma}_1>\sigma'_1$, or else $\tilde{\sigma}_1\sigma'_2\cdots\sigma'_k$ would be an occurrence of $\sigma$ inside the $\sigma$-stack, which is again prohibited. As a result, $\tilde{\sigma}_2$ must follow $\sigma'_1$ in $\gamma$, or else $\gamma$ would contain the occurrence $\tilde{\sigma}_1\tilde{\sigma}_2\sigma'_1$ of $231$. Therefore, $\gamma$ contains the subsequence:
$$
\tilde{\sigma}_1\sigma'_1\tilde{\sigma}_2\cdots\tilde{\sigma}_k.
$$
In particular, $\sigma'_1\tilde{\sigma}_2\cdots\tilde{\sigma}_k$ is an occurrence of $\sigma$, and $\sigma'_1$ is strictly to the right of $\tilde{\sigma}_1$. We can thus repeat the same argument with $\sigma'_1\tilde{\sigma}_2\cdots\tilde{\sigma}_k$ in place of $\tilde{\sigma}_1\tilde{\sigma}_2\cdots\tilde{\sigma}_k$, until we eventually find a contradiction\footnote{Equivalently, this yields a contradiction if $\tilde{\sigma}_1\cdots\tilde{\sigma}_k$ is supposed to be the rightmost occurrence of $\sigma$ in $\gamma$ with respect to lexicographical order of indices.}.
\end{proof}

Proposition~\ref{theorem_sorted_hat_231}, Proposition~\ref{proposition_sorted_necessary}, and Proposition~\ref{proposition_sorted_suff} lead to the following characterization of effective permutations.

\begin{corollary}\label{corollary_sorted_class}
Let $\sigma$ be a permutation of length two or more. Then $\sigma$ is not effective if and only if $\hat{\sigma}=1\oplus\alpha$, for some $\alpha\in\Sym(231)$.
\end{corollary}

If $\sigma$ is effective, then $\sorted(\sigma)=\Sym(231,\sigma)$ is completely determined (and enumerated). On the other hand, due to Corollary~\ref{corollary_sorted_class}, there are $\catalan_{n-1}=|\Sym_{n-1}(231)|$ patterns of length $n$ that are not effective, where $\catalan_n$ is the $n$-th Catalan number. Table~\ref{table_sorted_table} contains the first nine terms of the sequence $\lbrace |\sorted_n(\sigma)|\rbrace_{n\ge 1}$ when $\sigma$ is not effective and has length at most four. The only match in the OEIS~\cite{Sl} is $\sorted(21)$, which corresponds to the classical case of West-$2$-stack-sorting. The set $\sorted(21)$ was described and enumerated by Mireille Bousquet-Melou~\cite{BM2}. Enumerating the set $\sorted(\sigma)$ in all the other cases seems to be a much more challenging task, and remains an open problem.

\begin{table}
\begin{tabular}{lcr}
\begin{tabular}[t]{l}
\toprule
Effective permutations\\
\midrule
12\\
\midrule
123, 132, 231,321\\
\midrule
1234, 1243, 1324, 1342\\
1423, 1432, 2314, 2341\\
2413, 2431, 3142, 3214\\
3241, 3412, 3421, 4213\\
4231, 4312, 4321\\
\bottomrule
\end{tabular}
& &
\begin{tabular}[t]{ll}
\toprule
$\sigma$ & $|\sorted_n(\sigma)|,\ 1\le n\le 9$\\
\midrule
21   & A027432\\
213  & 1, 2, 4, 9, 22, 58, 161, 466, 1390\\
312  & 1, 2, 4, 8, 17, 40, 104, 291, 855\\
2134 & 1, 2, 5, 13, 34, 91, 252, 724, 2150\\
2143 & 1, 2, 5, 13, 35, 97, 277, 813, 2448\\
3124 & 1, 2, 5, 13, 34, 90, 244, 683, 1979\\
4123 & 1, 2, 5, 13, 33, 82, 203, 510, 1321\\
4132 & 1, 2, 5, 13, 34, 89, 234, 622, 1684\\
\bottomrule
\end{tabular}
\end{tabular}
\caption{Effective permutations, on the left. Number of $\sigma$-sorted permutations when $\sigma$ is not effective, on the right.}\label{table_sorted_table}
\end{table}

\begin{example}\label{example_123_machine}
Let $\sigma=123$. The authors of~\cite{CeClF} showed that $\Sort(123)=
1+\sum_{j=1}^{n-1}(n-j)\catalan_j$ by providing a detailed geometric description of $123$-sortable permutations. They also computed the generating function of $\Sort(123)$ via a bijection with a class of pattern-avoiding Schr\"oder paths. Here we use $\sigma$-sorted permutations as an alternative approach to the enumeration of $\Sort(123)$. By Corollary~\ref{corollary_sorted_class}, $123$ is effective and $\sorted(123)=\Sym(123,231)$. It is easy to see that any permutation $\gamma\in\Sym_n(123,231)$ uniquely decomposes as
$$
\gamma=
\identity_i\reverse\ominus
\bigl(\identity_j\reverse\oplus\identity_k\reverse\bigr),
$$
for some $j\ge 1$ and $i,k\ge 0$, with $i+j+k=n$. Now, let $\gamma$ be as above and let $\pi\in\bigl(\mapsigma{123}\bigr)^{-1}(\gamma)$.
Due to what proved in~\cite{CeClF} (namely Corollary 5.3, Corollary 5.6 and Theorem 5.7), either $\pi=\identity_n$, if $k=0$ and $\gamma=\identity_n\reverse$, or $\pi$ is uniquely obtained by:
\begin{enumerate}
\item Choosing a $213$-avoiding permutation of length $j\in\lbrace 1,\dots,n-1\rbrace$;
\item Inserting the string $j+1,j+2,\dots,j+k$ at the beginning;
\item Inserting the string $j+k+1,j+k+1,\dots,j+k+i$ immediately after the entry equal to $j-1$.
\end{enumerate}
In the latter case, $\pi$ is uniquely determined by picking a $213$-avoiding permutations of length $j$ and an integer $k\in\lbrace 1,\dots,n-j\rbrace$. Therefore:
$$
\fert{123}(\gamma)=
\begin{cases}
1, &\quad\text{ if } k=0;\\
(n-j)\catalan_j, &\quad\text{ if } k\ge 1,
\end{cases}
$$
and
$$
\Sort_n(123)=
\sum_{\gamma\in\sorted_n(123)}\fert{123}(\gamma)=
1+\sum_{j=1}^{n-1}(n-j)\catalan_j,
$$
as expected.
\end{example}

\section{Final remarks}\label{section_final_remarks}

In Section~\ref{section_bivincular_result} we have defined the binvincular pattern~$\xi=(132,\lbrace 0,2\rbrace,\emptyset)$ and determined which permutations $\sigma$ satisfy $\Sort(\sigma)\subseteq\Sym(\xi)$. Then, in Section~\ref{section_sorted_perm_fertilities} we have characterized effective permutations, that is, those where $\sorted(\sigma)=\Sym(231,\sigma)$. These two results, together with Theorem~\ref{theorem_suff_cond_class}~\cite{CeClF}, open the door to a classification of $\sigma$-machines. More in detail, we have that:

\begin{itemize}
\item[(Cls)] \textbf{Theorem~\ref{theorem_suff_cond_class}}: $\Sort(\sigma)$ is a class if and only if $\hat{\sigma}\ge 231$. In this case, we have $\Sort(\sigma)=\Sym(132)$, if $\sigma\ge 231$, and $\Sort(\sigma)=\Sym\bigl(132,\sigma\reverse\bigr)$, if $\sigma\not\ge 231$.
\item[($\xi$)] \textbf{Theorem~\ref{theorem_bivincular_pattern_gen_result}}: $\Sort(\sigma)\not\subseteq\Sym(\xi)$ if and only if $\hat{\sigma}\not\ge 231$ and $\sigma\reverse\ge\xi$.
\item[(Eff)] \textbf{Corollary~\ref{corollary_sorted_class}}: $\sigma$ is not effective if and only if $\hat{\sigma}=1\oplus\alpha$, for some $\alpha\in\Sym(231)$.
\end{itemize}

Permutations can be partitioned according to which of the above conditions they satisfy, and the arising classification is illustrated in Table~\ref{table_classification} for permutations of length three and four.

\begin{table}
\begin{center}
\def\arraystretch{1.375}
\begin{tabular}{p{16mm}p{16mm}cccp{42mm}}
\toprule
hyp 1 & hyp 2 & Cls & Eff & $\xi$ & Permutations\\
\midrule
$\hat{\sigma}\ge 231$ & $\sigma\ge 231$
& $\checkmark$ & $\checkmark$ & $\checkmark$
& 1342, 2341, 2431, 3142\newline 3241, 4231\\
& $\sigma\not\ge 231$
& $\checkmark$ & $\checkmark$ & $\checkmark$
& 321, 3214, 4213, 4312\newline 4321\\
\hline
$\hat{\sigma}\not\ge 231$\newline $\hat{\sigma}_1\neq 1$ & $\sigma\not\ge 231$
& $\times$ & $\checkmark$ & $\checkmark$
& 123, 132, 1234, 1243\newline 1324, 1423, 1432\\
& $\sigma\reverse\not\ge\xi$\newline $\sigma\ge 231$
& $\times$ & $\checkmark$ & $\checkmark$
& 2314, 2413\\
& $\sigma\reverse\ge\xi$
& $\times$ & $\checkmark$ & $\times$
& 231, 3412, 3421\\
\hline
$\hat{\sigma}\not\ge 231$\newline $\hat{\sigma}_1=1$
&
& $\times$ & $\times$ & $\checkmark$
& 213, 312, 2134, 2143\newline 3124 4123, 4132\\
\bottomrule
\end{tabular}
\end{center}
\caption{Classification of $\sigma$-machines. The leftmost two columns denote the hypotheses satisfied by $\sigma$. The columns Cls, Eff and $\xi$ contain a checkmark if $\Sort(\sigma)$ is a class, $\sigma$ is effective and $\Sort(\sigma)\subseteq\Sym(\xi)$, respectively. Permutations of length three and four are listed in the last column.}\label{table_classification}
\end{table}

At the top of the table we find the simplest family of $\sigma$-machines, that is, those where $\hat{\sigma}\ge 231$. Here $\Sort(\sigma)$ is a permutation class by Theorem~\ref{theorem_suff_cond_class}, $\Sort(\sigma)$ is contained in $\Sym(\xi)$ by Theorem~\ref{theorem_bivincular_pattern_gen_result}, and $\sigma$ is effective by Corollary~\ref{corollary_sorted_class}. More precisely, the first row contains those permutations where $\sigma\ge 231$ and $\Sort(\sigma)=\Sym(132)$, whereas the second row contains those where $\sigma\not\ge 231$ and $\Sort(\sigma)=\Sym\bigl(132,\sigma\reverse\bigr)$.

The remaining part of the table contains permutations where $\hat{\sigma}$ avoids $231$ and $\Sort(\sigma)$ is not a class. In the middle we find those that are effective. Equivalently, due to Corollary~\ref{corollary_sorted_class} (and the assumption $\hat{\sigma}\not\ge 231$), those where $\hat{\sigma}_1\neq 1$. Now, if $\sigma$ avoids $231$, then necessarily $\sigma\reverse$ avoids $\xi$ (since $132$ is the classical pattern underlying $\xi$); thus $\sigma$ is effective and $\sorted(\sigma)=\Sym(231,\sigma)$. On the other hand, if $\sigma\ge 231$ and $\sigma\reverse$ avoids $\xi$, then $\sigma$ is still effective, but $\sorted(\sigma)=\Sym(231)$. In both cases, we have $\Sort(\sigma)\subseteq\Sym(\xi)$. Finally, if $\sigma\reverse\ge\xi$, then $\sigma\ge 231$, $\sorted(\sigma)=\Sym(231)$, and $\Sort(\sigma)$ is not a subset of $\Sym(\xi)$.

The bottom part of the table contains the permutations that are not effective. In this case, we have $\hat{\sigma}\not\ge 231$ and $\hat{\sigma}_1=1$; in particular, this implies that $\sigma\reverse$ avoids $\xi$, and thus $\Sort(\sigma)\subseteq\Sym(\xi)$. Otherwise, if $\sigma\reverse\ge\xi$, then we would have $\sigma\ge 231$ and, since $\hat{\sigma}_1=\sigma_2=1$, $\hat{\sigma}$ would contain $231$ as well, which is a contradiction.

To sum up, the last two rows of the table seem to contain the permutations~$\sigma$ whose corresponding $\sigma$-machines are the most challenging to be studied: the penultimate row contains those where $\Sort(\sigma)\not\subseteq\Sym(\xi)$, while the last row contains those that are not effective. For what concerns permutations of length three, three cases have been solved so far~\cite{CeClF,CeClFS}:

\begin{itemize}
\item $321$, where $\Sort(321)=\Sym(123,132)$ is a class;
\item $123$ and $132$, which are effective permutations with $\Sort(\sigma)\subseteq\Sym(\xi)$.
\end{itemize}

The remaining three are currently open, namely:

\begin{itemize}
\item $231$, which is effective but has $\Sort(231)\not\subseteq\Sym(\xi)$;
\item $213$ and $312$, which satisfy $\Sort(\sigma)\subseteq\Sym(\xi)$ but are not effective.
\end{itemize}

The first ten terms of the arising sequences can be found in Table~\ref{table_unsolved_patterns}. Among them, the most promising case is $\sigma=312$. The OEIS~\cite{Sl} suggests that $|\Sort_n(312)|$ could be equal to the number of ascent sequences of length $n$ that avoid $201$~\cite{DS}, denoted as $A(201)$. These have been shown~\cite{CC} to be in bijection with the set $F(3412)$ of Fishburn permutations avoiding $3412$. A numerical investigation suggests the equidistribution of the following pairs of statistics:
\begin{itemize}
\item[-] (Left-to-right maxima, Right-to-left maxima) on $\Sort(312)$;
\item[-] (Left-to-right maxima, Left-to-right minima) on $F(3412)$;
\item[-] (Right-to-left minima, number of zeros) on $A(201)$.
\end{itemize}
Despite the strong evidence, the problem of finding bijections between $\Sort(312)$ and the other two sets remains open.

Regarding the effective permutation $\sigma=231$, where $\sorted(231)=\Sym(231)$, a formula to compute the $231$-fertility of any $231$-avoiding permutation would allow us to count $\Sort(231)$ by replicating the approach adopted in Example~\ref{example_123_machine} for $\sigma=123$.

\begin{table}
\centering
\def\arraystretch{1.1}
\begin{tabular}{llr}
\toprule
$\sigma$ & $|\Sort_n(\sigma)|, 1\le n\le 10$ & \textbf{OEIS}\\
\midrule
213 & 1, 2, 5, 16, 62, 273, 1307, 6626, 35010, 190862 & ?\\
231 & 1, 2, 6, 23, 102, 496, 2569, 13934, 78295, 452439 & ?\\
312 & 1, 2, 5, 15, 52, 201, 843, 3764, 17659, 86245 & A202062\\
\bottomrule
\end{tabular}
\caption{Enumerative data for unsolved permutations of length three.}
\label{table_unsolved_patterns}
\end{table}





\end{document}